\documentclass[12pt,reqno]{amsart}
\usepackage{amssymb,amsmath,amsgen,amsopn,amsxtra,amsbsy,amscd}
\def\({\left(}
\def\){\right)}
\def\Nx{\nabla}

\def\eb{\varepsilon}

\def\al{\alpha}
\def\rw{\rightarrow}
\def\Om{\Omega}

\def\la{\lambda}
\def\La{\Lambda}

\def\R {\mathbb{R}}

\def\V {{\mathcal V}}

\def\V {{\mathcal V}}
\def\F {{\mathcal F}}

\def \p {\partial}

\def\pt{\partial_t}

\def \and{\qquad\text{and}\qquad}

\def\Nx{\nabla}
\def \p {\partial}
\def\Dx{\Delta}
\newcommand{\be}{\begin{equation} }
\newcommand{\ee}{\end{equation} }

\newtheorem{proposition}{Proposition}[section]
\newtheorem{theorem}[proposition]{Theorem}

\newtheorem{lemma}[proposition]{Lemma}
\theoremstyle{definition}

\newtheorem{remark}[proposition]{Remark}

\numberwithin{equation}{section}

\def \no#1#2#3 {{\bf #1} (#3), #2.}
\def \eds#1#2#3 {#1, #2, #3.}

\title[Feedback Stabilization of Damped Nonlinear Wave  Equations]
{Global Stabilization of the Navier-Stokes-Voight and  the damped nonlinear wave equations by  finite number of  feedback controllers}
\author[]
{Varga K. Kalantarov and Edriss S. Titi}

\address{(V.K.Kalantarov) Department of mathematics, Ko{\c c} University,
\newline\indent Rumelifeneri Yolu, Sariyer 34450\newline\indent
Sariyer, Istanbul, Turkey
\newline\indent Institute of Matematics and Mechanics,
\newline\indent National  Academy of Sciences of Azerbaijan
\newline\indent Baku, Azerbaijan }
\email{vkalantarov@ku.edu.tr}

\address{(E.S.Titi) Department of Mathematics, Texas A\&M University, 3368 TAMU,
\newline\indent College Station, TX 77843-3368, USA.  {\bf ALSO},
\newline\indent Department of Computer Science and Applied Mathematics, Weizmann
\newline\indent Institute of Science, Rehovot 76100, Israel.}
\email{titi@math.tamu.edu}
  \email{edriss.titi@weizmann.ac.il}

 
\begin{document}

\begin{abstract}{In this paper we introduce  a finite-parameters feedback control algorithm for stabilizing solutions of the  Navier-Stokes-Voigt equations, the
strongly damped nonlinear wave equations and the nonlinear wave equation with nonlinear damping term,  the  Benjamin-Bona-Mahony-Burgers equation and  the KdV-Burgers equation. This algorithm capitalizes on the fact that such infinite-dimensional dissipative dynamical systems posses finite-dimensional long-time behavior which is represented by, for instance, the  finitely many determining parameters of their long-time dynamics, such as determining Fourier modes, determining volume elements, determining nodes , etc..The algorithm utilizes these finite parameters in the form of feedback control to stabilize the relevant solutions. For the sake of clarity, and in order to fix ideas, we focus in this work on the case of low Fourier modes feedback controller, however, our results and tools are equally valid for using other feedback controllers employing other spatial coarse mesh  interpolants.  }
\end{abstract}
\date{21 May 2017}
\maketitle

\vspace{0.2cm}
\begin{center} {\it  This work is dedicated  to the memory of Professor Igor Chueshov.}
\end{center}
\vspace{0.2cm}
\section{Introduction}\label{s0}

The stabilization problem of  nonlinear parabolic equations, such as the Navier-Stokes equations and other related equations of hydrodynamics, the linear and nonlinear wave equations have been intensively investigated  by various authors (see, e.g., \cite{BaSl}, \cite{Bar1},\cite{BaCo}, \cite{CaKeTi}, \cite{ChLa}, \cite{CoTr}, \cite{FuKa}, \cite{Har},\cite{LaTr}, \cite{Mar}, \cite{SmCeKr}, \cite{Z1} and references therein).
Some of these works have been specifically devoted to the problem of feedback stabilization by controllers depending only on finitely many parameters for nonlinear parabolic equations and related systems, such as the reaction-diffusion equations, the Navier-Stokes equations, the Ozeen equations, the phase-field equations, the Kuramoto-Sivashynsky equations, etc... (see, e.g., \cite{AzTi},\cite{Ba},\cite{Bar2}-\cite{BaWa}, \cite{CaKeTi} \cite{Che},\cite{KaTi},\cite{LuTi}, \cite{Mun}, \cite{RoTe} and references therein). This large body of work relies, whether explicitly or implicity, on the fact that the asymptotic in time behavior  of such infinite-dimensional dissipative dynamical systems is governed by finitely many degrees of freedom. This fact was first established for the two-dimensional Navier--Stokes equations in the pioneer works  of C.\ Foias and G.\ Prodi \cite{FoPr},
and of O.A.\ Ladyzhenskaya \cite{La1}. Specifically, the authors proved  that the  long-time behavior of solutions of 2D Navier-Stokes equations and the trajectories in the global attractor of are determined by the dynamics of  finitely many, but large enough number, of  Fourier modes. This seminal work   triggered a subsequent investigation concerning the finite-dimensional asymptotic in time behavior of solutions of disspative nonlinear PDE\rq{}s (see, e.g.,  \cite{BaVi},
\cite{FMRT}-\cite{FoTe}, \cite{Ha},\cite{JT93}, \cite{La2},\cite{Te} and references therein). Indeed, it was shown, for instance,  that the long-time behavior of solutions of 2D Navier-Stokes equations, reaction-diffusion equations, complex Ginzburg-Landau equations, Kuramoto-Sivashinsky equation, 1D damped semilinear wave equations and a number of other disspative eqations can be determined by finitely many nodes and volume elements, etc... (see, e.g., \cite{CJT1}, \cite{CJT2}, \cite{FoTi}, \cite{FMRT}, \cite{HaRa}, \cite{JT93},  \cite{KaTi1} and references therein).   The general concept of determining interpolant operators (determining functionals) was introduced in \cite{CJT1},\cite{CJT2} and which enabled the authors to present a unified approach for investigating these determining parameters. A further extension of the applicability of this unified approach for studying  the  long-time behavior of various  nonlinear dissipative PDE\rq{}s was developed in, e.g., \cite{CJT2} , \cite{Ch} - \cite{ChLa} and refrences therein.\\

In this paper we study the problem of global stabilization of the solutions of the initial boundary value problems for the 3D
Navier-Stokes-Voigt (NSV) equations
\be\label{0nsv}
\p_t v-\nu \Dx v-\al ^2\Dx \p_tv+(v\cdot \Nx)v+\Nx p=h(x), \ \Nx\cdot v=0 , \ x \in \Om, t
 \in \R^+,
\ee
and the following damped nonlinear dispersive equations:\\
the strongly damped nonlinear wave equation
\be\label{0str}
 \pt^2u-\Dx u -b\Dx \pt u-\la u+f(u)=h(x), \ x \in \Om, t>0,
 \ee
the nonlinear wave equation with a nonlinear damping term
\be\label{0nd}
\p_t^2u+g(\p_t u)-\Dx u +f(u)=h(x), \ x \in \Om, t>0,
\ee
the Benjamin-Bona-Mahoni-Burgers (BBMB) equation,
\be\label{0bbmb}
\p_t v-\p_x^2\p_tv+f(v) \p_xv -\p_x^2v=h(x), \ \ x\in (0,1), t>0
\ee
under the Dirichlet boundary condition
and the Korteweg-de Vries-Burgers (KdVB) equation
\be\label{0kdv}
\p_t v+\p_x^3v+v \p_xv -\p^2_xv=h(x), \ \ x\in (0,1), t>0,
\ee
under periodic boundary conditions. Here $\Om\subset \R^3$ is a bounded domain with sufficiently smooth boundary, $\al$ and $b$ are given positive parameters, $h\in L^2$ is a given function, $f(\cdot), g(\cdot)$ are given nonlinear terms satisfying natural growth conditions guaranteing the global existence of corresponding initial boundary value problems. \\

Our main goal is to show that any arbitrary  given solution of the initial boundary value problem for each of equations \eqref{0nsv} -\eqref{0kdv} can be stabilized  by using  a feedback controller depending only on finitely many large spatial-scale parameters, such as the  low Fourier modes or other finite rank spatial interpolant operators that are based on coarse mesh spatial measurements.
A common feature of these equations is that the semigroups generated by initial boundary value problems are asymptotically compact  semigroups which have finite-dimensional global attractors in corresponding phase spaces (see e.g. \cite{CKP}, \cite{Ha}, \cite{K86}-\cite{KaZe}, \cite{Pr}, \cite{Te}, \cite{WaYa} and references therein). In \cite{AzTi}, \cite{KaTi} and \cite{LuTi} the authors introduced a feedback control algorithm based on the above mentioned unified approach employing finitely many parameters for the global stabilization of  solutions  for  a number nonlinear   nonlinear parabolic equations and to the  damped nonlinear wave equations. In this work we develop further this algorithm  and extend/demonstrate its applicability to a larger  class of dissipative PDEs. However, for the sake of clarity, and in order to fix ideas, we focus in this work on the case of low Fourier modes feedback controller. Nonetheless, it is worth stressing that  our results and tools are equally valid when using other feedback controllers that are employing other spatial coarse mesh  interpolants and determining functionals. This will be the subject of forthcoming publication. \\

Throughout this paper  we will use the following notations
\begin{itemize}
\item $(\cdot,\cdot)$ and $\|\cdot\|$ will denote the inner product and the norm for both $L^2(\Om)$ and $\left(L^2(\Om)\right)^3$.  For
$u=(u_1,u_2,u_3),v=(v_1,v_2,v_3)\in \left(L^2(\Om)\right)^3$, $(u,v):=\int\limits_\Om\sum\limits_{j=1}^3u_jv_jdx$.
\item $H$ will denote the closure of $\V:=\left\{u\in \left(C_0^\infty (\Om)\right)^3: \Nx \cdot u=0\right\}$ in
$\left(L^2(\Om)\right)^3$, and $V$ is the closure of $\V$ in $\left(H_0^1(\Om)\right)^3$.
\item $\dot{H}_{per}^k(0,1):=\left\{\phi \in H^k_{per}(0,1): \int_0^1 \phi(x)dx=0, k=1,2,\cdots\right\}$.
\item $w_1,w_2,\cdots, w_n,\cdots$ will denote the eigenfunctions of the Stokes operator and the Laplace operator $-\Dx (-\frac{d^2}{dx^2})$ under the homogeneous Dirichlet boundary condition and of the operator $-\frac{d^2}{dx^2}$ in $\dot H^2_{per}(0,1)$ corresponding to eigenvalues $\la_1,\la_2,\cdots,\la_n,\cdots$.

\end{itemize}
In what follows we are using the following inequalities:
\begin{itemize}
\item {\it Young\rq{}s inequality}
For each $a,b>0$ and  $\epsilon>0$
\be\label{Young}
ab \le \epsilon \frac{a^p}p+ \frac1{\epsilon^{q/p}}\frac{b^q}q,
\ee
where $p,q>0$ and $\frac1p+\frac1q=1.$

\item {\it Poincar\'e-Friedrichs  inequality}
\be\label{PF}
\|u\|^2\le \la_1^{-1}\|\Nx u\|^2, \ \ \forall u\in  H_0^1 (\dot H_{per}^1, V)
\ee
and the inequality
\be\label{PFN}
\sum\limits_{k=N+1}^\infty|(u, w_k)|^2\le \la_{N+1}^{-1}\|\Nx u\|^2, \ \ \forall u\in  H_0^1 (\dot H_{per}^1, V),
\ee
where $\la_1$ is the first and $\la_{N+1}$ is the $(N+1)$th eigenvalue of the operator $-\Dx$ (or $-\frac{d^2}{dx^2} $) under the Dirichlet or periodic boundary conditions.
\item {\it The $1D$ Agmon  inequality}
\be\label{Ag}
\max_{x\in [0,1]}|u(x)|^2\le c_0 \|u\|\|u'\|, \ \ \forall u\in H_0^1(0,1) (\dot H_{per}^1(0,1)).
\ee
\item {\it The $3D$ Ladyzhenskaya inequality}
\be\label{Lad3}
\|u\|^2_{L^4(\Om)}\le b_0\|u\|^{\frac12}\|\Nx u\|^{\frac32}, \ \ \forall u\in V,
\ee
where $\Om \subseteq \R^3$.

\item {\it The $1D$ Gagliardo-Nirenberg inequality}
\be\label{GN}
\|u^{(j)}\|_{L^p(0,1)}\le \beta\|u\|^{1-\al}\|u^{(m)}\|^\al, \ \ \forall u\in H^2(0,1)\cap H_0^m(0,1), (\dot H_{per}^m(0,1)).
\ee
where $p>2, m=1,2, \ \ \frac jm\le \al\le 1, \ \ \al =\left(\frac12+j-\frac1p\right)m^{-1}.$
\end{itemize}
\section{Stabilization employing finitely many Fourier modes\\ for Navier-Stokes-Voight  equations}
\subsection{ 3D Navier-Stokes-Voigt equations}
In this section we study the problem of global stabilization of 3D Navier-Stokes-Voigt equations by finitely many Fourier modes.
Suppose that $v$ is a given weak solution of the problem
\begin{equation}\label{v1}
\begin{cases}
\p_t v-\nu \Dx v-\al^2 \Dx \p_tv+(v\cdot \Nx)v+\Nx p=h, \ x \in \Om, t
 \in \R^+,\\
\Nx\cdot v=0 , \ x \in  \Om, t \in \R^+; v\Big|_{\p \Om}=0, \ \
\ t\in \R^+,\\
v(x,0)=v_0(x),\ x \in \Om ,
\end{cases}
\end{equation}
in a bounded domain $\Om\subset \R^3$ with sufficiently smooth
boundary $\partial \Om$, $v_0\in V$ is a given initial data and $h\in L^2(\Om)$ is a given source term. The equatin \eqref{v1} was introduced by A.P.Oskolkov in \cite{Osk1} as a model
of motion of linear viscoelastic non-Newtonian fluids. This equations were also proposed in \cite{CaLuTi} as a regularisation , for small values of the parameter $\al$, of the 3D Navier-Stokes equations.
Here  $v=v(x,t)$ is the velocity vector
field, $p$ is the pressure, $\nu>0$  is the kinematic viscosity,  and
$h$ is a given force field. The positive parameter $\al$ is
characterizing the elasticity of the fluid in the sense
that $\frac{\al^2}{\nu}$ is a characteristic relaxation time scale
of the viscoelastic material.
Our aim is to stabilize any strong solution $v$. That is for any initial data $u_0$ we will show  that the solution $u$ of of the feedback control
system system converges to  function $v$ as $t\rw \infty$, provided $N$ is large enough depending on physical parameters of the equation \eqref{v1}. That is  we will show that each solution of the system, for $N$ large enough,
\begin{equation}\label{u1}
\begin{cases}
\p_t u-\Dx (\nu u+\al^2  \p_tu)+(u\cdot \Nx)u+\Nx \tilde p =-\mu \sum\limits_{k=1}^N(u-v,w_k)w_k+h, x \in \Om, t
 >0,\\
\Nx\cdot u=0 , \ x \in  \Om, t \in \R^+; u\Big|_{\p \Om}=0, \ \
\ t\in \R^+,\\
u(x,0)=u_0(x), \ \ x\in \Om,
\end{cases}
\end{equation}
where $u_0 \in V$ is given, will tend to solution of the problem \eqref{v1} in $V$ norm as $t\rw \infty.$
Existence and uniqueness of solution $u\in C(\R^+, V)$  can be shown exactly the same way as it is done for the problem  \eqref{v1}. It is well known that the problem \eqref{v1} generates a continuous bounded dissipative semigroup $S(t): V\rw V$ with and absorbing ball $B_0(r_0)\subset V$ (see e.g. \cite{K86}, \cite{KaTi}.)
 In fact taking inner product of \eqref{v1} in $L^2(\Om)$ with $v$ we get
\be\label{NSL2}
\frac12\frac d{dt}\left[\|v(t)\|^2+\al^2 \|\Nx v(t)\|^2\right]+\nu \|\Nx v(t)\|^2=(h,v(t)).
\ee
Employing Young\rq{}s inequality \eqref{Young} and the Poincar\'e-Friedrichs inequality \eqref{PF} we obtain
$$
\frac d{dt}\left[\|v(t)\|^2+\al^2 \|\Nx v(t)\|^2\right]+
\nu \|\Nx v(t)\|^2\le
\frac 1{\nu\la_1}\|h\|^2,
$$
which implies that
$$
\frac d{dt}\left[\|v(t)\|^2+\al^2\|\Nx v(t)\|^2\right]+
d_0\left[\|v(t)\|^2+\al \|\Nx v(t)\|^2\right]\le
\frac 1{\nu\la_1}\|h\|^2,
$$
where $d_0=\frac\nu2\min\{\la_1, \frac{1}{\al^2}\}$.
From this inequality we obtain that
$$
\|v(t)\|^2+\al ^2\|\Nx v(t)\|^2\leq e^{-d_0 t}\left[\|v_0\|^2+\al \|\Nx v_0\|^2\right]+\frac 1{\nu\la_1d_0}\|h\|^2.
$$
That is each solution $u\in C(\R^+; H_0^1(\Om))$ of the problem the norm $\|\Nx v(t)\|$ is uniformly bounded on $\R^+,$ moreover
 there exists $t_0>0$ such that
\be\label{abs1}
\|v(t)\|^2+\al^2 \|\Nx v(t)\|^2\leq \frac 2{\nu\la_1d_0}\|h\|^2, \ \ t\geq t_0.
\ee
So the ball $B_0(r_0)\subset V$ with the radius
\be\label{radius}
r_0=\frac 2{\nu\la_1d_0\al^2}\|h\|^2
\ee
is an absorbing ball for the problem \eqref{v1}.\\

By using standard Galerkin method it can be shown that system \eqref{u1} has also a global strong solution.
Moreover it has an absorbing ball in $V$. Really, multiplying \eqref{u1} by $u$ in $L^2(\Om)$ ve obtain
\begin{multline*}\label{absu}
\begin{split}
\frac12\frac d{dt}[\|u(t)\|^2&+\al^2 \|\Nx u(t)\|^2]+\nu \|\Nx u(t)\|^2=\\&-\mu \sum_{k=1}^N|(u(t),w_k)|^2+\mu \sum_{k=1}^N(u(t),w_k)(v(t),w_k)+(h,u(t))\\&
\le \frac \mu4 \sum_{k=1}^N|(v(t),w_k)|^2+\|h\|\|u(t)\|\le \frac \mu4 \|v(t)\|^2+\frac \nu2 \|\Nx u(t)\|^2
+\frac1{2\nu\la_1}\|h\|^2.
\end{split}
\end{multline*}
Thanks to \eqref{abs1} we have
$$
\frac d{dt}\left[\|u(t)\|^2+\al ^2\|\Nx u(t)\|^2\right]+d_0\left[\|u(t)\|^2+\al^2 \|\Nx u(t)\|^2\right]\le \frac\mu{2\la_1}r_0+\frac1{\nu\la_1}\|h\|^2.
$$
From this inequality we deduce that
$$
\|u(t)\|^2+\al^2 \|\Nx u(t)\|^2\leq e^ {-d_0t}\left[\|u_0\|^2+\al^2 \|\Nx u_0\|^2\right]  +   \frac1{d_0}\left[\frac\mu{2\la_1}r_0+\frac1{\nu\la_1}\|h\|^2 \right]
$$
Hence the ball $B_0(r_1)\subset V $ with radius $r_1=\frac2{d_0\al^2}\left[\frac\mu{2\la_1}r_0+\frac1{\nu\la_1}\|h\|^2 \right]$ is an absorbing ball for problem \eqref{u1}.

\begin{theorem}\label{Fcon1}
Suppose that
\be\label{uv1}
\mu \ge
 C_1(\nu)r_0^2b_0^4, \ \ \ \nu-4\la_{N+1}^{-1} C_1(\nu)r_0^2b_0^4>0,
\ee
where $C_1(\nu)=\frac14\left(\frac3{2\nu}\right)^3$, $b_0$ is a constant in the Ladyzhenskaya  inequality  \eqref{Lad3} and $r_0$ is the radius of the absorbing ball of the problem \eqref{v1}.\\
Then there exists $t_0>0$  such that
the following inequality holds true
\be\label{uv2}
\|\Nx u(t)-\Nx v(t)\|\leq k_0e^{-\kappa (t-t_0)}\|\Nx u(t_0)-\Nx v(t_0)\|, \ \ t\ge t_0,
\ee
where $k_0=1+\frac1{\al^2\la_1}, \ \ \kappa=\frac\nu4\min\{\la_1,\al^{-2}\}.$
\end{theorem}
\begin{proof}
If $v$ is a solution of the problem \eqref{v1} and $u$ is a solution of system \eqref{u1} ,
then the function $z=u-v$ is a solution of the system
\begin{equation}\label{w1}
\begin{cases}
 \p_tz-\nu \Dx z-\al^2 \Dx \p_tz+(u\cdot \Nx)z+(z\cdot \Nx)v+\Nx \pi =-\mu \sum\limits_{k=1}^N(z,w_k)w_k, x \in \Om, t
 >0,\\
\Nx\cdot z=0 , \ x \in  \Om, t \in \R^+; \ \ z\Big|_{\p \Om}=0, \ \ t>0,
\\
z(x,0)=z_0(x),\ x \in \Om ,
\end{cases}
\end{equation}
where $z_0:=u_0-v_0, \pi:=\tilde p- p$.\\

Inner product of \eqref{w1} in $L^2(\Om)$ with $z$ gives
\be\label{w4}
\frac12\frac d{dt}\left[\|z(t)\|^2+\al^2\|\Nx z(t)\|^2\right]+\nu \|\Nx z(t)\|^2+((z(t)\cdot \Nx)v(t),z(t))=-\mu \sum_{k=1}^N|(z(t),w_k)|^2.
\ee
Thanks to the Ladyzhenskaya inequality \eqref{Lad3} and the Young inequality \eqref{Young} we have
\begin{multline}\label{w6}
\begin{split}
|((z(t)\cdot \Nx)v(t), z(t))|& \le \|z(t)\|_{L^4(\Om)}^{2}\|\Nx v(t)\|
\le b_0\|z(t)\|^ {\frac12}\|\Nx z(t)\|^ {\frac32}\|\Nx v(t)\|\\& \le \frac \nu2\|\Nx z(t)\|^2+C_1(\nu) b_0^4\|z(t)\|^2\|\Nx v(t)\|^4,
\end{split}
\end{multline}
where $C_1(\nu)=\frac14\left(\frac3{2\nu}\right)^3$.\\
Employing the fact that $\|\Nx v(t)\|^4\le r_0^2$ for $t\ge t_0$ and \eqref{w6} we obtain from \eqref{w4} the inequality
\be\label{w7}
\frac12\frac d{dt}\left[\|z(t)\|^2+\al^2\|\Nx z(t)\|^2\right]+\frac\nu2 \|\Nx z(t)\|^2=-\mu \sum_{k=1}^N|(z(t),w_k)|^2+C_1(\nu)r_0^2b_0^4\|z(t)\|^2.
\ee
Since
$$
\|z(t)\|^2=\sum_{k=1}^N|(z(t),w_k)|^2+\sum_{k=N+1}^\infty|(z(t),w_k)|^2\le \sum_{k=1}^N|(z(t),w_k)|^2+\la_{N+1}^{-1}\|\Nx z(t)\|^2
$$
we obtain from \eqref{w7} the inequality
\begin{multline}\label{z7}
\begin{split}
\frac12\frac d{dt}\left[\|z(t)\|^2+\al^2\|\Nx z(t)\|^2\right]&+\left(\frac\nu2- \la_{N+1}^{-1}C_1(\nu)r_0^2b_0^4\right)\|\Nx z(t)\|^2\\&=\left(-\mu+C_1(\nu)r_0^2b_0^4\right) \sum_{k=1}^N|(z(t),w_k)|^2.
\end{split}
\end{multline}
Taking into account  conditions \eqref{uv1}, \eqref{z7} and the Poincar\'e-Friedrichs inequality \eqref{PF} we obtain from \eqref{z7}
the inequality
$$
\frac d{dt}\left[\|z(t)\|^2+\al^2\|\Nx z(t)\|^2\right]+\kappa\left[\|z(t)\|^2+\al^2\|\Nx z(t)\|^2\right]\leq 0, \ \forall t\geq t_0,
$$
where $\kappa=\frac\nu4\min\{\la_1,\al^{-2}\}$.
Integrating the last inequality over the interval $(t_0,t)$ we get
$$
\|\Nx u(t)-\Nx v(t)\|\leq k_0\|\Nx u(t_0)-\Nx v(t_0)\|e^{-\kappa(t-t_0)}, \ \ \forall t\ge t_0.
$$
Hence the inequality \eqref{uv2} holds true.
\end{proof}

\section{ Stabilization employing finitely
many Fourier modes for  damped nonlinear wave equations.}
\subsection{Benjamin-Bona-Mahony-Burgers equation}
Let $v\in C(\R^+; H^1_0(0,1))$ be a strong solution of the generalized Benjamin-Bona-Mahony(BBMB) equation
\be\label{bbm1}
\p_t v-\p_x^2\p_tv+f(v) \p_xv -\p_x^2v=h(x), \ \ x\in (0,1), t>0,
\ee
satisfying the initiial conditiion
\be\label{bbm2}
v(x,0)=v_0(x), \ \ x\in (0,1)
\ee
and the Dirichlet  boundary conditions.
\be\label{bbm3}
v(0,t)=v(1,t)=0, \ \  \forall t>0.
\ee

Here $h\in L^2(0,1)$ is a given source term and $f(\cdot)\in C^1(\R)$ is a given function.\\
First energy equality for this problem is the equality
$$
\frac12\frac d{dt}\left[\|v(t)\|^2+\|\p_xv(t)\|\right]^2+\|\p_xv(t)\|^2=(h,v(t)).
$$
Applying the Poincar\'e-Friedrichs inequality \eqref{PF} after some manipulations we get
\begin{multline*}
\begin{split}
\frac d{dt}\left[\|v(t)\|^2+\|\p_xv(t)\|\right]^2+\la_1\|v(t)\|^2+\|\p_xv(t)\|^2&\le 2\|h\|\|v(t)\|\\&\le
 \frac{\la_1}2\|v(t)\|^2+\frac2{\la_1}\|h\|^2.
\end{split}
\end{multline*}
This inequlity implies the inequality
\be\label{sim}
\frac d{dt}\left[\|v(t)\|^2+\|\p_xv(t)\|^2\right]+\kappa_1\left[\|v(t)\|^2+\|\p_xv(t)\|\right]^2\le \frac2{\la_1}\|h\|^2,
\ee
where $\kappa_1=\min\{\frac{\la_1}2,1\}$.\\
Therefore there exists $t_0>0$ such that
\be\label{bbm4}
\|v(t)\|^2+\|\p_xv(t)\|^2\le \frac{4}{\la_1\kappa_1}:=R_1, \ \ \ \ t\ge t_0.
\ee
We propose the following feedback system,
to stabilize the solution $v(x, t)$  of the problem \eqref{bbm1}-\eqref{bbm3}:
\be\label{bbu1}
\begin{cases}
\p_t u-\p_x^2\p_tu+u\p_xu -\p_x^2u=-\mu\sum\limits_{k=1}^N(u-v,w_k)w_k+h(x), \ \ x\in (0,1), t>0,\\
u(x,0)=u_0(x), \ \ x\in (0,1),\\
u(0,t)=u(1,t)=0, \ \ \ t>0.
\end{cases}
\ee
Our aim is to show that for a given $u_0\in H_0^1(0,1)$ and properly choosen $\mu$ and $N$ the function
$\|\p_xu(t)-\p_xv(t)\|$ tends to zero as $t\rw \infty$ with an exponential rate.\\
Multiplication of \eqref{kdu1} by $u$ in $L^2$ gives
\begin{multline*}
\begin{split}
\frac12\frac d{dt}\left[\|u\|^2+\|\p_xu\|^2\right] &+\|\p_xu\|^2=-\mu\sum_{k=1}^N|(u,w_k)|^2+\mu\sum_{k=1}^N(u,w_k)(v,w_k)+(h,u)\\&\le
\frac{\mu}2\sum_{k=1}^N|(v,w_k)|^2+\frac{\la_1}2\|u\|^2+\frac 2{\la_1}\|h\|^2.
\end{split}
\end{multline*}
From this inequality taking into account \eqref{bbm4}, similar to \eqref{sim}, we obtain  for $\La(t):=\|u(t)\|^2+\|\p_xu(t)\|^2$ the following inequality
$$
\frac d{dt} \La(t) +a_1\La(t)\le \frac{\mu R_1}2+\frac 2{\la_1}\|h\|^2, \ \ \forall t\ge t_0.
$$
The last inequality implies that there exists some $t_1\ge t_0$
such that
\be\label{bbm4a}
\|u(t)\|^2+\|\p_xu(t)\|^2\le R_2, \ \ \ \ t\ge t_1,
\ee
where $R_2=\mu R_1+\frac 2{\la_1}\|h\|^2.$\\

It is clear that the function $z=u-v$ is a solution of the problem
\be\label{bbz1}
\begin{cases}
\p_t z+\p^2_x\p_tz+f(u) \p_xz +(f(u)-f(v))\p_x v-\p_x^2z=-\mu\sum\limits_{k=1}^N(z,w_k)w_k, \ \ x\in (0,1), t>0,\\
z(x,0)=u_0(x)-v_0(x), \ \ x\in (0,1),\\
z(0,t)=z(1,t)=0, \ \ t>0.
\end{cases}
\ee
Multiplying \eqref{bbz1} in $L^2$ by $z$ we obtain
\begin{multline}\label{bbz4}
\begin{split}
\frac12\frac d{dt}&\left[\|z\|^2+\|\p_xz\|^2\right]+\|\p_xz\|^2\\&+
(f(u)\p_x z,z)+f\rq{}(\theta u+(1-\theta)v)z^2, \p_x v)=-\mu\sum_{k=1}^N|(z,w_k)|^2, \ \ \theta \in (0,1).
\end{split}
\end{multline}
Thanks to the  the Sobolev inequality
\be\label{int}
\|z\|_{L^\infty(0,1)}\le \|\p_xz\|_{L^2(0,1)}
\ee
we have
\begin{multline}\label{bbz5}
\begin{split}
\Big|(f(u)\p_x z,z)&+f\rq{}(\theta u+(1-\theta)v)z^2, \p_x v)\Big|\\& \le D_1\|\|z\|\|\p_xz\|+D_2\|\p_xv\|\|z\| ^2\le \frac12\|\p_xz\|^2+\left( \frac12 D_1^2+D_2\|\p_x v\|\right)\|z\|^2,
\end{split}
\end{multline}
where
$$
D_1=\max\limits_{|s|\le R_2}|f(s)|, \ \ \ D_2=\max\limits_{|s|\le R_1+R_2}|f\rq{}(s)|.
$$
Thanks to \eqref{bbz5}, \eqref{bbm4} and \eqref{bbm4a} we obtain from \eqref{bbz4} the inequality
\be\label{bbz4a}
\frac12\frac d{dt}\left[\|z\|^2+\|\p_xz\|^2\right]+\frac12 \|\p_xz\|^2=-\mu\sum_{k=1}^N|(z,w_k)|^2+\left(\frac12 D_1^2+D_2\sqrt{R_1}\right)\|z\|^2, \ \forall t\ge t_2.
\ee
Suppose that $N$  and
$\mu$ are  large enough such that
\be\label{Nmu}
\la_{N+1}^{-1}\left(D_1^2+2D_2\sqrt{R_1}\right)\le \frac12 ,\ \ \mbox{and} \ \ \mu \ge \frac12 D_1^2+D_2\sqrt{R_1}.
\ee
Then \eqref{bbz4a} yields
$$
\frac d{dt}\left[\|z(t)\|^2+\|\p_xz(t)\|^2\right]+\frac12\|\p_xz(t)\|^2\le 0.
$$
Applying Poincar\'e-Friedrichs inequality \eqref{PF} we get
$$
\frac d{dt}\left[\|z(t)\|^2+\|\p_xz(t)\|^2\right]+ \frac14\min\{1,\la_1\}\left[\|z(t)\|^2+\|\p_xz(t)\|^2\right]\le 0, \ t\ge t_1.
$$
Hence
\be\label{bbex}
\|z(t)\|^2+\|\p_xz(t)\|^2\le e^{-a_0(t-t_1)}\left(\|z(t_1)\|^2+\|\p_xz(t_1)\|^2\right), \ \ \ \ \forall t\ge t_1,
\ee
where $a_0=\frac14\min\{1,\la_1\}.$
So we proved the following theorem
\begin{theorem}\label{bbT}
Suppose \eqref{Nmu} are satisfied, then there exists a number $t_1>0$ such that

\be\label{bbex}
\|\p_xu(t)-\p_x v(t)\|^2\le e^{-a_0(t-t_1)} \left(\|\p_xu(t_1)-\p_x v(t_1)\|^2\right), \ \ \ \ \forall t\ge t_1.
\ee
\end{theorem}
\begin{remark} Statement of the Theorem \ref{bbT} holds also for solutions of BBMB  equation under the periodic boundary conditions
$$
u(x,t)=u(x+1,t), \ \  \forall x\in \R, t>0, \ \ \int_0^1u(x,t)dx=0.
$$
\end{remark}
\subsection{ Korteweg-de Vries-Burgers equation}
In this section  we consider stabilization of Korteveg-de Vries-Burgers equation.
Let $v\in C(\R^+; \dot H_{per}^1(0,1))$ be a strong solution of the equation
\be\label{kdv1}
\p_t v+\p_x^3v+v \p_xv -\p^2_xv=h(x), \ \ x\in \R, t>0,
\ee
satisfying the initial condition
\be\label{kdv2}
v(x,0)=v_0(x), \ \ x\in \R
\ee
and the periodic boundary conditions.
\be\label{kdv3}
v(x,t)=v(x+1,t),\ \ \forall x\in \R, t>0.
\ee
Here $v_0\in \dot{H}_{per}^1(0,1)$ is a given intial function, $h\in L_{per}^2(0,1)$ is a given source term with $\int_0^1h(x)dx=0.$\\
There have been many studies on the long-time behavior of solutions of the KdVB equation  (see e.g. \cite{Ji}, \cite{LuLu}, \cite{Zh}).
Our goal here is to show that any solution of the initial boundary value problem \eqref{kdv1}-\eqref{kdv3}
can be stabilised by using feedback controller employing finitely many Fourier modes
as observables and controllers.\\
To this end we use the first energy equality
$$
\frac12\frac d{dt}\|v(t)\|^2+\|\p_xv(t)\|^2=(h,v(t))
$$
to infer that there exists $t_0>0$ such that
\be\label{L2es}
\|v(t)\|^2\le \frac2{\la^2_1}\|h\|^2:=\rho_1, \ \ \forall t\ge t_0.
\ee
Next we show that the problem \eqref{kdv1}-\eqref{kdv3} generates a continuous bounded dissipative semigroup also in $\dot{H}_{per}^1(0,1)$.
Multiplication of \eqref{kdv1} in $L^2$ by $v^2$ gives:
$$
\frac13\frac d{dt}(v^3,1)-2(v\p_xv,\p^2_xv)-(v^2,\p_x^2v)=(h,v^2)
$$
Using equation \eqref{kdv1}
$$
-v\p_xv=\p_tv+\p_x^3v-\p_x^2v-h
$$
we get
$$
\frac d{dt}\left[ \|\p_xv\|^2-\frac13(v^3,1)\right]+2\|\p^2_xv\|^2=-(v^2,\p_x^2v) -(h,v^2)-2(h,\p_x^2v).
$$
This equality inplies that
$$
\frac d{dt}\left[ \|\p_xv\|^2-\frac13(v^3,1)\right]+\frac32\|\p^2_xv\|^2\le 2\|v\|^2_{L^4(0,1)}+5\|h\|^2.
$$

Thanks to the Poincar\'e-Friedrichs  inequality \eqref{PF} and Young inequality  \eqref{Young} we have
\begin{multline}\label{kdv4}
\begin{split}
\frac d{dt}\left[ \|\p_xv\|^2-\frac13(v^3,1)\right]&+\la_1\left[ \|\p_xv\|^2-\frac13(v^3,1)\right]+\frac12\|\p^2_xv\|^2
\le\frac{\la_1}3|(v^3,1)|\\&+2\|v\|_{L^4}^{2}+5\|h\|^2\le 3\|v\|_{L^4}^{4}+5\|h\|^2+\left(\frac{\la_1}3\right)^4+1.
\end{split}
\end{multline}
Next we us the Gagliardo-Nirenberg inequality \eqref{GN}  with $p=4$
$$
 \|u\|^4_{L^4}\le \beta ^4\|\p_x^2u\|^{\frac12}\|u\|^{\frac72}.
$$
 and the Young inequality \eqref{Young} we get
$$
3\|v\|_{L^4}^{4}\le 3\beta^4\|\p_x^2v\|^{\frac12}\|v\|^{\frac72}\le \frac12\|\p_x^2v\|^2
+\frac1{2^{1/3}}\left(3\beta^4\|v\|^{7/2}\right)^{4/3}.
$$

Employing last four estimates we deduce from \eqref{kdv4} the inequality
\be\label{kdv5}
\frac d{dt}\left[ \|\p_xv\|^2-\frac13(v^3,1)\right]+\left[\|\p_xv\|^2-\frac13(v^3,1)\right]\le
\rho_2, \ \ \ \ \  t\ge t_0,
\ee
where
$$\rho_2:=\frac1{2^{1/3}}\left(3\beta^4\rho_1^{7/4}\right)^{4/3}+5\|h\|^2+\left(\frac{\la_1}3\right)^4+1.$$
  From \eqref{kdv5} we infer that there exists $t_1\ge t_0$ such that
$$
\|\p_xv\|^2-\frac13(v^3,1)\le 2\rho_2 \ \ \ \forall t\ge t_1\ge t_0.
$$
Since
$$
\frac13|(v^3,1)|\le \frac{\beta^3}3\|\p_xv\|^{\frac12}\|v\|^{\frac52}\le \frac14\|\p_xv\|^2+\frac45\left(\frac13\beta^3\|v\|^{\frac52}\right)^{\frac54}
$$
 and thanks to \eqref{L2es} we get
\be\label{kdv6}
\|\p_xv(t)\|^2 \le \frac83\rho_2+\frac{16}{15}\left(\frac13\beta^3\rho_1^{\frac52}\right)^{\frac54}:=M_0,  \ \ \ \forall t\ge t_1.
\ee
We propose the following feedback system,
to stabilize the solution $v(x, t)$  of the problem \eqref{kdv1}-\eqref{kdv3}:
\be\label{kdu1}
\begin{cases}
\p_t u+\p_x^3u+u \p_xu -\p_x^2u=-\mu \sum\limits_{k=1}^N(u-v,w_k)w_k+h(x), \ \ x\in (0,1), t>0,
\\
u(x,0)=u_0(x), \ \ x\in (0,1),
\\
u(x,t)=u(x+1,t), \ \ \ \forall x\in \R, t>0.
\end{cases}
\ee
Our aim is to show that for given $u_0\in \dot H_{per}^1(0,1)$ and properly choosen $\mu$ and $N$  to be large enough the function
$\|u(t)-v(t)\|$ tends to zero as $t\rw \infty$ with an exponential rate.\\
First we obtain some uniform estimates for solutions of the problem \eqref{kdu1}. The uniform estimate of $L^2$ norm of solution we obtain from the first energy equality
$$
\frac12\frac{d}{dt}\|u\|^2+\|\p_xu\|^2=-\mu\sum_{k=1}^N|(u,w_k)|^2+\mu\sum_{k=1}^N(u,w_k)(v,w_k)+(h,u).
$$
Thanks to the Cauchy-Schwarz  inequality and Poincar\'e-Friedrichs  inequality \eqref{PF} we obtain the inequality
$$
\frac{d}{dt}\|u\|^2+\la_1\|u\|^2=\frac\mu2\sum_{k=1}^N|(v,w_k)|^2+\la_1\|h\|^2\le \frac \mu2\rho_1+\frac1{\la_1}\|h\|^2.
$$
which implies, thanks to Gronwall\rq{}s inequality, that there exists $t_2\ge t_1$, depending on $u_0$, such that
$$
\|u(t)\|^2\le \la_1^{-1}\left(\mu \rho_1+\la_1^{-1}\|h\|^2\right)=:M_1, \ \ \forall t\ge t_2.
$$
To get the uniform estimate for $\|\p_xu(t)\|$ we multiply \eqref{kdu1} by $-2\p_x^2u$ in $L^2$ and obtain the equality
\be\label{Nes1}
 \frac{d}{dt}\|\p_xu\|^2+\int_0^1(\p_xu)^3dx+2\|\p_x^2u\|^2=-2\mu\sum_{k=1}^N\la_k\left[(u,w_k)^2-(u,w_k)(v,w_k)\right]-2(h,\p_x^2u).
\ee
Employing the 1D Gagliardo-Niranberg inequality \eqref{GN}, and the Young inequality \eqref{Young} we get the estimate
\be\label{gn3}
|((\p_xu)^3,1)|\le \beta^3\|u\|^{\frac{11}4}\|\p_x^2u\|^{\frac14}\le \frac12 \|\p_x^2u\|^2+\beta_1|u\|^{22/7},
\ee
where $\beta_1=\beta^{24/7}$ and
\be\label{gn2}
|(u,w_k)(v,w_k)|\le (u,w_k)^2+\frac14(v,w_k)^2, \ \ 2|(h,\p_x^2u)|\leq \frac12\|\p_x^2u\|^2+2\|h\|^2.
\ee
By using \eqref{gn3} and \eqref{gn2} we obtain from \eqref{Nes1}
\be\label{Nes2}
 \frac{d}{dt}\|\p_xu\|^2+\|\p_x^2u\|^2\le \frac\mu2\sum_{k=1}^N\la_k(v,w_k)^2+\beta_1\|u\|^{22/7}+2\|h\|^2,
\ee
Finally due to the Poincar\'e-Friedrichs  inequality \eqref{PF}, and since
$$
\sum_{k=1}^N\la_k(v,w_k)^2\le\|\p_xv(t)\|^2\le M_0, \ \ \|u(t)\|^2\le M_1, \ \  \forall t\ge t_2
$$
we deduce from \eqref{Nes2}  that there exists $t_3\ge t_2$, depending on $u_0$, such that
\be\label{Nes0}
 \|\p_xu(t)\|^2\le M_2, \ \ \forall t\ge t_3\ge t_2,
\ee
where $M_2=\mu M_0+\beta_1M_1^{11/7}+2\|h\|^2.$

It is clear that the function $z=u-v$ is a solution of the problem
\be\label{kdz1}
\begin{cases}
\p_t z+\p_x^3z+u \p_xz +z\p_xv-\p_x^2z=-\mu\sum\limits_{k=1}^N(z,w_k)w_k, \ \ x\in (0,1), t>0,\\
z(x,0)=u_0(x)-v_0(x), \ \ x\in (0,1),\\
z(x,t)=z(x+1,t),  \ \ \forall x\in \R, t>0.
\end{cases}
\ee
Multiplying \eqref{kdz1} in $L^2$ by $z$ we obtain
\be\label{zes1}
\frac12\frac d{dt}\|z(t)\|^2+(u\p_xz,z) +(z^2(t),\p_xv)+\|\p_xz(t)\|^2=-\mu\sum_{k=1}^N|(z(t),w_k)|^2.
\ee
By using the Gagliardo-Nirenberg  inequality \eqref{GN} and the Cauchy - Schwarz inequality and the estimates \eqref{kdv6} and \eqref{Nes0} we get
\begin{multline}\label{zes2}
\begin{split}
|(u\p_xz,z)|=\frac12|(z^2,\p_x u)|&\leq \frac12\|\p_xu\|\|z\|^2_{L^4} \le  \frac12\beta^2\|\p_xu\|\|z\|^{\frac32}\|\p_xz\|^{\frac12}\\&\le\frac12\|\p_xz\|^2 +\left(\frac12\beta^2\|\p_xu\|\right)^{4/3}\|z\|^2, \ \forall t\ge t_3.
\end{split}
\end{multline}
Similarly we obtain
\be\label{zes3}
|(z^2,\p_xv)|\le  \|v\|_{L^4}^2 \|\p_xv\|\le \beta^2\|z\|^ {1/2}\|\p_xz\|^{3/2}
\le\frac12\|\p_xz\|^2+\frac14\left(\|\p_xv\|\beta^2\right)^4\|z\|^2,  \forall t\ge t_3.
\ee
Thanks to \eqref{zes2} and \eqref{zes3} we  get from \eqref{zes1} the inequality
$$
\frac d{dt}\|z(t)\|^2+\|\p_xz(t)\|^2\le-2\mu\sum_{k=1}^N(z(t),w_k)^2+M_{3}\|z\|^2, \ \ \forall t\ge t_3,
$$
where $M_{3} :=2\beta^4\la_1^{-\frac34}\left(M_0^2+\frac14 M_2^2\right)$.
Using the inequality \eqref{PFN}  and the fact that $$\|z(t)\|^2=\sum\limits_{k=1}^N(z(t),w_k)^2+\sum\limits_{k=N+1}^\infty(z(t),w_k)^2, \ \ \forall t \ge t_3,$$ we get
$$
\frac d{dt}\|z(t)\|^2+\|\p_xz(t)\|^2\le\left(-2\mu+M_3\right)\sum_{k=1}^N(z(t),w_k)^2+M_3\la_{N+1}^{-1}\|\p_xz(t)\|^2, \ \ \forall t\ge t_3.
$$
Hence, if $N$ and $\mu$ are large enough, such that $M_3\la_{N+1}^{-1}\le\frac12$ and $M_3\le2\mu$, we have
$$
\frac d{dt}\|z(t)\|^2+\frac {\la_1}2\|z(t)\|^2\le0, \ \ \forall t\ge t_3.
$$
Thus
$$
\|z(t)\|\le z(t_3)e^{-(t-t_3)}, \ \ \forall t\ge t_3.
$$

So we proved the following theorem
\begin{theorem}
Suppose that  $v\in \dot H_{per}^1(0,1)$ is a given  weak solution of  problem \eqref{kdv1}-\eqref{kdv3} and the following conditions hold true:
$$
2\beta^4\la_1^{-\frac34}\left(M_0^2+\frac14 M_2^2\right)\le 2\mu, \ \ \mbox{and} \ \ \la_{N+1}^{-1}\left(2\beta^4\la_1^{-\frac34}\left(M_0^2+\frac14 M_2^2\right)\right)\le \frac12,
$$
where $M_0, M_2$ are defined in \eqref{kdv6} and \eqref{Nes0}. Then
$$
\|u(t)-v(t)\|\rw 0 \ \ \mbox{with an exponential rate, as} \ \ t\rw \infty,
$$
where $u(t)\in \dot H_{per}^1(0,1)$  is an arbitrary solution  of the feedback control system.
\end{theorem}
\subsection{Strongly damped nonlinear wave equation}

In this section we consider  the initial boundary value problem for $3D$ strongly damped nonlinear wave equation:

\noindent \be\label{foz1}
\begin{cases}
 \pt^2v-\Dx v-b\Dx \pt v-\la v+f(v)=0, \ x \in \Om, t>0, \\
 v=0, \
x \in \p \Om, \ t>0, \\
v(x,0)=v_0(x), \ \pt v(x,0)=v_1(x), \ \
x \in \Om, \ t>0,
\end{cases}
\ee
where $b>0$ is a given number, $\Om\subset \R^3$ is a bounded domain with a smooth boundary $\p \Om,f: C^1(\R\rw \R)$ is a given function which satisfies the condition
\begin{equation}\label{1.int}
-m_0+ a|s|^{p} \leq f'(s)\leq m_0(1+|s|^{p} ),\ \ \ \ \
\forall s \in\R^1
\end{equation}
with some $ m_0>0,a>0, p\geq 2$. The existence of a unique weak (energy) solution of the problem \eqref{foz1}, i.e. a  function
$$
v\in  L^\infty([0,T],H^1_0(\Om)\cap L^{p+1}(\Om))\cap
W^{1,\infty}([0,T],L^2(\Om))\cap W^{1,2}([0,T],H^1(\Om)), \ \forall t>0,
$$
which satisfies the equation in the sense of distributions,  is  established in \cite{KaZe}.
Here we consider the feedback control problem for stabilizing the solution $v$ of the  $3D$ strongly damped nonlinear wave equation,  based on finitely many Fourier modes,
 i.e. , we consider the  feedback system of the following form:
\be\label{fo1}\begin{cases}
 \pt^2u-\Dx u -b\Dx \pt u-\la u+f(u)=-\mu
\sum\limits_{k=1}^N(z +\p_tz,w_k)w_k, \ x \in \Om, t>0, \\
u=0, \
x \in \p \Om, \ t>0, \\
 u(x,0)=u_0(x), \ \pt u(x,0)=u_1(x), \ \
x \in \Om, \ t>0,
\end{cases}\ee
where  $z=u-v, \ \ \mu>0$ is a
given numbers,   $w_1,w_2,..., w_n,...$ is the set of orthonormal (in
$L^2(\Om)$) eigenfunctions  of the Laplace operator $-\Dx$ under the
homogeneous Dirichlet's boundary condition, corresponding to
eigenvalues
$$0<\la_1\leq \la_2 \cdots \leq \la_n\leq \cdots .$$
In what follows we will use the following lemma:
\begin{lemma}\label{Da} (see for instance \cite{Z1}) If the function $f$ satisfies condition
\eqref{1.int}  then
\be\label{Dams}
[f(u_1)-f(u_2)].(u_1-u_2)\ge -\la|v|^2+d_0(|u_1|^{p}+|u_2|^{p})|v|^2,
\ee
for some positive $\la,d_0$ and $p\geq 2$.
\end{lemma}
Our main result in this section is the following theorem:

\begin{theorem}\label{TF} Suppose that $\mu$ and $N$  are large enough such that
\be\label{cF1} \nu \geq (2a+3b^2/4)\la_{N+1}^{-1}, \ \ \mbox{and} \ \  \mu \geq a+3b^2/4. \ee
 Then the following decay estimate holds true
  \be\label{estW}
\|\pt z(t)\|^2+\|\Nx z(t)\|^2\leq E_0e^{-\frac b2
t}, \ee
where
\begin{multline*}
\begin{split}
E_0:&=\frac12\|u_1\|^2+\frac \nu 2\|\Nx u_0\|^2\\&
+(\frac{b^2}4-\frac
a2)\|u_0\|^2+ \frac1p\int_\Om |u_0(x)|^pdx
+\frac\mu2\sum_{k=1}^N(u_0,w_k)^2+\frac { b}2(u_0,u_1).
\end{split}
\end{multline*}
\end{theorem}
\begin{proof}
It is clear that the function $z=u-v$ satisfies
\be\label{fov1}
\begin{cases}
\pt^2z-\Dx z -b\Dx \pt z+f(u)-f(v)=-\mu
\sum\limits_{k=1}^N(z+\p_t z,w_k)w_k, \ x \in \Om, t>0, \\
 z=0, \ \mbox{for} \
x \in \p \Om, \ t>0, \\
 z(x,0)=z_0(x), \ \pt z(x,0)=z_1(x), \ \
x \in \Om, \ t>0,
\end{cases}
\ee
where $v_0=u_0-v_0, z_1=u_1-v_1$.

Let $P=(-\Dx)^{-1}$ be the inverse of the Laplace operator under the homogeneous Dirichlet boundary condition. First we multiply the equation \eqref{fov1} by
$P\p_tv$ and integrate over $x\in\Omega$
\begin{multline}\label{fcn1}
\begin{split}
\frac d{dt}\left[\|P^{\frac12}\p_t z\|^2+ \| z\|^2+\mu\sum\limits_{k=1}^N\la_k^{-1}(z,w_k)^2\right]&+2b\|\p_tz\|^2
+2(f(u)-f(v),P\p_tz)\\&+2\mu\sum\limits_{k=1}^N\la_k^{-1}(\p_tz,w_k)^2=0.
\end{split}
\end{multline}
Multiplying \eqref{fov1} by $z$ and integrating over $\Omega$ we get
\begin{multline}\label{fcn2}
\begin{split}
\frac d{dt}\left[\frac{ b}2\|\nabla z\|^2+ (z,\p_tz)+\frac{\mu}2\sum\limits_{k=1}^N(z,w_k)^2\right]&- \|\p_t z\|^2+\|\nabla z\|^2\\&=
-(f(u)-f(v),z)-\mu\sum\limits_{k=1}^N(z,w_k)^2.
\end{split}
\end{multline}
Now we multiply \eqref{fcn2} by a positive parameter $\eb >0$    (to be chosen below) and add  to \eqref{fcn1} to obtain:

\begin{multline}\label{fcn3}
\begin{split}
\frac d{dt}E_\eb(t)+ (2b-\eb )\|\p_t z\|^2&+2(f(u)-f(v),P\p_tz)+2\mu\sum\limits_{k=1}^N\la_k^{-1}(\p_tz,w_k)^2\\&+
\eb(f(u)-f(v),z)+\eb\|\nabla z\|^2  +\eb \mu\sum\limits_{k=1}^N(z,w_k)^2=0,
\end{split}
\end{multline}
where
\begin{multline*}
\begin{split}
E_\eb(t):=\|P^{\frac12}\p_t z(t)\|^2+ \| z(t)\|^2+\mu\sum\limits_{k=1}^N(\la_k^{-1}+\frac{\eb}2)(z(t),w_k)^2&+
\frac{\eb b}2\|\nabla z(t)\|^2\\&+ \eb(z(t),\p_tz(t)).
\end{split}
\end{multline*}
It is easy to see that if $0<\eb\leq \frac b2$ then
\be\label{Eb}
E_\eb(t)\ge \frac12\|P^{\frac12}\p_t z\|^2+ \| z\|^2+\mu\sum\limits_{k=1}^N(\la_k^{-1}+\frac{\eb}2)(z,w_k)^2+
\frac{\eb b}4\|\nabla z\|^2.
\ee
Employing the interpolation inequality
$$
\|P\p_tz\|^2\leq c_0\|\p_tz\|\|P^{\frac12}\p_tz\|
$$
and the condition \eqref{1.int} we can estimate the term $2(f(u)-f(v),P\p_tz)$ as follows
\begin{multline}\label{fest1}
\begin{split}
|2(f(u)-f(v),P\p_tz)|\leq  \eb_1(|u|+|v|)^p,|z|)^2&+\eb_1\|z\|^2 +\eb_1\|\p_tz\|^2 \\& +C(\eb_1)\|P^{\frac12}\p_tz\|^2.
\end{split}
\end{multline}
Since
$$
\left((|u|+|v|)^p,|z|\right)^2\leq C\left(\int_\Om(|u^p|+|v|^p)dx\right)\left((|u|+|v|)^p,|z|^2\right)
$$
and the integral $\int_\Om(|u(x,t)|^p+|v(x,t)|^p)dx$ is uniformly bounded we obtain from \eqref{fest1} that
\begin{multline}\label{fest2}
\begin{split}
|2(f(u)-f(v),P\p_tz)| \leq  \eb_1C(|u|^p+|v|^p),|z|^2)&+\eb_1\|z\|^2 +\eb_1\|\p_tz\|^2 \\& +C(\eb_1)\|P^{\frac12}\p_tz\|^2,
\end{split}
\end{multline}
where $C(\eb_1)=\frac{C}{\eb_1^3}.$

By using the inequalities \eqref{fest2} and \eqref{Dams} we deduce from \eqref{fcn3} the inequality:
\begin{multline}\label{fest2a}
\begin{split}
\frac d{dt}E_\eb(t)&+ (2 b-\eb -\eb_1)\|\p_t z\|^2+(d_0\eb -\eb_1C)(|u|^p+|v|^p),|z|^2)\\&
+\eb\|\nabla z\|^2 -(\eb_1+\eb \la)\| z\|^2 \\&+\mu\sum\limits_{k=1}^N\left[\eb(z,w_k)^2+2\la_k^{-1}(\p_tz,w_k)^2\right]
 -C(\eb_1)\|P^{\frac12}\p_tz\|^2\le0.
\end{split}
\end{multline}
By choosing  $\eb_1=\frac{d_0\eb}C$ and $0<\eb\leq \min\{\frac b2,\frac{bC}{C+d_0}\}$ in \eqref{fest2} we get
\begin{multline*}
\begin{split}
\frac d{dt}E_\eb(t)+ b\|\p_t z\|^2&
+\eb\|\nabla z\|^2 -(\eb_1+\eb \la)\| z\|^2 \\&+\mu\sum\limits_{k=1}^N\left[\eb(z,w_k)^2+ 2\la_k^{-1}(\p_tz,w_k)^2\right]
 -C(\eb_1)\|P^{\frac12}\p_tz\|^2\le0.
\end{split}
\end{multline*}
Let us rewrite the last inequality as follows
\begin{multline}\label{fest3}
\begin{split}
\frac d{dt}E_\eb(t)+ b\|\p_t z\|^2&
+\eb\|\nabla z\|^2 +(\mu \eb-\eb_1 -\eb\la)\sum\limits_{k=1}^N(z,w_k)^2\\&-(\eb_1 +\eb\la)\sum\limits_{k=N+1}^\infty(z,w_k)^2
+(2\mu-C(\eb_1))\sum\limits_{k=1}^N(\p_tz,w_k)^2\la_k^{-1}\\&-C(\eb_1)\sum\limits_{k=N+1}^\infty(\p_tz,w_k)^2\la_k^{-1}\leq 0.
\end{split}
\end{multline}
Finally by choosing $\mu$ and $\la_{N+1}$ large enough we infer from \eqref{fest3}
the following inequality
\be\label{fest3}
\frac d{dt}E_\eb(t)+ \frac b2\|\p_t z\|^2
+\frac\eb2\|\nabla z\|^2 +\frac12\mu \eb\sum\limits_{k=1}^N(z,w_k)^2\leq 0.
\ee
Employing the last inequality and \eqref{Eb} we can show that there exists some $\delta>0$ depending on $\eb$ such that 
$$
\frac d{dt}E_\eb(t)+\delta E_\eb(t)\le0.
$$
The last inequality implies that
$$\|P^{\frac12}\p_t z(t)\|^2+\|\nabla z(t)\|^2$$
tends to zero with an exponential rate, as $t\rw\infty$.

\section{Wave equation with nonlinear damping term}
In this section we consider the initial boundary value problem for a semilinear wave equation with nonlinear damping:
\be\label{ndm1}
\begin{cases}
\p_t^2v+g(\p_tu)-\Dx u +f(u)=0, \ x \in \Om, t>0,
\\
v=0, \
x \in \p \Om, \ t>0, \\
 v(x,0)=v_0(x), \ \pt v(x,0)=u_1(x), \ \
x \in \Om, \ t>0,
\end{cases}
\ee
where $f\in C(\R\rw \R)$ is a given function which satisfies the  condition  \eqref{1.int}, $g\in C(\R\rw\R)$ is a given function which satisfies the conditions
\be\label{gc1}
g(0)=0, \ \ |g(u_1)-g(u_2)|\le a_0(1+|u_1|^{m}+|u_2|^{m})|u_1-u_2|, \ \forall u_1,u_2\in \R
\ee
and
\be\label{gc2}
[g(u_1)-g(u_2)].(u_1-u_2)\ge a_1|u_1-u_2|^2+a_2|u_1-u_2|^{m+2}, \ \forall u_1,u_2\in \R.
\ee
It is well known that, under the above conditions, the stationary problem
\be\label{ndm2}
\begin{cases}
 -\Dx \phi+f(\phi)=0, \ x \in \Om,\\
 \phi=0, x \in \p\Om,
\end{cases}
\ee
corresponding to \eqref{ndm1}, has finitely many solutions.

Next we will show that the system \eqref{ndm1} can be globally stabilized to a given stationary state $\phi$
 also by using a feedback controller involving finitely many Fourier modes. More precisely we will show that
all solutions of the following feedback control problem
\be\label{ndm1}
\begin{cases}
\p_t^2 u+g(\p_t u)-\Dx  u +f( u)=-\mu\sum\limits_{k=1}^N( u-\phi,w_k)w_k, \ x \in \Om, t>0
\\
u=0, \
x \in \p \Om, \ t>0, \\
 u(x,0)=u_0(x), \ \pt u(x,0)=u_1(x), \ \
x \in \Om, \ t>0,
\end{cases}\ee
tend  in the energy norm to the stationary state $\phi$, as $t\rw \infty.$ For global existence and long time behavior of solutions of \eqref{ndm1} see e.g. \cite{Har}, \cite{Lions},\cite{Mar}, \cite{Pr}.\\

It is clear that the function $z=u-\phi$ satisfies
\be\label{ndm2}
\begin{cases}
\p_t^2z+g(\p_tz)-\Dx z+f(u)-f(\phi)=-\mu\sum\limits_{k=1}^N(z,w_k)w_k, \ x \in \Om, t>0,
\\
z=0, \
x \in \p \Om, \ t>0, \\
 z(x,0=z_0(x), \ \pt z(x,0)=z_1(x), \ \
x \in \Om, \ t>0,
\end{cases}\ee
where $z=u-\phi, z_0=u_0-\phi, z_1=u_1$, $m$ is a given number such that $m>0$  if $n=1,2$ and $m\le \frac4{n-2}$ if $n\ge 3.$ \\
Our aim is to show that under some restrictions on $\mu, N$ the function $z$ tends to zero as $t\rw \infty.$

First we multiply \eqref{ndm2} by $\p_tz$ and integrate over $\Om$:
$$
\frac12\frac d{dt}\left[\|\p_tz\|^2+\|\Nx z\|^2+\mu\sum_{k=1}^N(z,w_k)^2\right]+(f(u)-f(\phi),\p_tz)+(g(\p_tz),p_tz)=0.
$$
Since
$$
(f(u)-f(\phi),\p_t z)=\frac d{dt}[(F(u),1)-(f(\phi),z)]
$$
and
$$
(F(u),1)=\int_0^1f(\phi+sz)+(F(\phi),1)
$$
we have
\be\label{ndm3a}
\frac d{dt}\left[\frac12\|\p_tz\|^2+\frac12|\Nx z\|^2+\frac{\mu}2\sum_{k=1}^N(z,w_k)^2+\F(z)\right]+(g(\p_tz),p_tz)=0,
\ee
where
$$
\F(z)=\int_0^1(f(\phi+sz)-f(\phi),z)ds.
$$
Due to the condition \eqref{1.int}
$$
\F(z)\ge-\frac{m_0}2\|z\|^2+\frac{d_0}{p+2}\int_G|z|^{p+2}dx.
$$
Therefore if
\be\label{mucn}
\mu\ge m_0 \ \ \mbox{and} \ \ \la_{N+1}\ge 2\mu m_0
\ee
then
\begin{multline}\label{ndm3}
\begin{split}
E(t):=\frac12\|\p_tz\|^2+\frac12\|\Nx z\|^2&+\frac{\mu}2\sum_{k=1}^N(z,w_k)^2+\F(z)\\& \ge \frac12 \|\p_tz\|^2+\frac14\|\Nx z\|^2+\frac{d_0}{p+2}\int_G|z|^{p+2}dx.
\end{split}
\end{multline}
We will need also the following inequality which we obtain by integration of \eqref{ndm3a} over the interval $(0,t)$ and using the condition \eqref{gc2}
\be\label{ndE}
E(t)+a_1\int_0^t \|\p_\tau z(\tau)\|^2d\tau +a_2\int_0^t\int_G|\p_\tau z(x,\tau)|^{m+2}dxd\tau \le E(0).
\ee
Next we multiply \eqref{1.int} by $z$:
\be\label{ndm4}
\frac d{dt}(z,\p_tz)=\|\p_tz\|^2-\|\Nx v\|^2-(f(z+\phi)-f(\phi),z)-(g(\p_tz),z)-\mu \sum_{k=1}^N(z,w_k)^2.
\ee
Employing the equality
$$
-\frac12\|\Nx z\|^2=\frac32\|\p_tz\|^2+\frac{\mu}2\sum_{k=1}^N(z,w_k)^2+\F(z) -E(t)
$$
we get from \eqref{ndm4}:
\begin{multline}\label{ndm5}
\frac d{dt}(z,\p_tz)=\frac32\|\p_tz\|^2-\frac12\|\Nx z\|^2+[\F(z)-(f(z+\phi)-f(\phi),v)]-(g(\p_tz),z)\\ -\frac\mu2 \sum_{k=1}^N(z,w_k)^2-E(t).
\end{multline}
Thanks to  condition \eqref{1.int} we have
\begin{multline*}
\begin{split}
\F(z)&-(f(z+\phi)-f(\phi),z) =\int_0^1(f(\phi+sz)-f(\phi+z),z)ds\\ & =-\int_0^1(f(\phi+z)-f(\phi+sz),z)ds\le
\frac{m_0}2\|v\|^2-\frac{d_0}{p+2}\int_G|v|^{p+2}dx.
\end{split}
\end{multline*}
Utilizing the last inequality and  conditions \eqref{mucn} we obtain from \eqref{ndm5} that
$$
\frac d{dt}(z,\p_tz)\leq \frac32\|\p_tz\|^2-(g(\p_tz),z)-E(t).
$$
Integrating lthe ast inequality and employing  condition \eqref{gc1} we get
\begin{multline}\label{ndm5}
\begin{split}
\int_0^tE(\tau)d\tau & \le (z(0),\p_t z(0))-(v(t),\p_tz(t))+\frac32\int_0^t\|\p_\tau z(\tau)\|^2 d\tau\\ & +a_0\int_0^t\|\p_\tau z(\tau)\|\| z(\tau)\|d\tau
+a_0\int_0^t\int_G|\p_\tau z(x,\tau)|^{m+1}|z(x,\tau)|dx d\tau.
\end{split}
\end{multline}
Thanks to \eqref{ndm3} and \eqref{ndE} we have
\be\label{ndm6}
\Big|(z(0),\p_t z(0))-(z(t),\p_tz(t))+\frac32 \int_0^t\|\p_\tau z(\tau)\|^2 d\tau\Big|\le C_1.
\ee
On the other hand, employing the Cauchy - Schwarz inequality and \eqref{ndE} we get
\be\label{ndm7}
\int_0^t\|\p_\tau z(\tau)\|\| z(\tau)\|d\tau \le C_2t^{\frac12},
\ee
Thanks to \eqref{ndE}, the H\"{o}lder inequality and the continuous embedding of $H_0^1(G)$ into $L^{m+2}(G)$ we estimate the last term on the right-hand side of \eqref{ndm5}:
\begin{multline}\label{ndm8}
\begin{split}
\int_0^t&|\p_\tau z(x,\tau)|^{m+1}|z(x,\tau)|dx d\tau\\ &\le \left(\int_0^t\int_G|\p_\tau z(x,\tau)|^{m+2}dx d\tau\right)^{\frac{m+1}{m+2}}
\left(\int_0^t\int_G| z(x,\tau)|^{m+2}dx d\tau\right)^{\frac{1}{m+2}}\\ &
\le C \left(\int_0^t\int_G|\p_\tau z(x,\tau)|^{m+2}dx d\tau\right)^{\frac{m+1}{m+2}}\left(\int_0^t\| \Nx z(\tau)\|^{m+2} d\tau\right)^{\frac{1}{m+2}}\le C_3 t^{\frac1{m+2}}.
\end{split}
\end{multline}
Hence due to \eqref{ndm6}-\eqref{ndm8} we obtain from \eqref{ndm5} the inequality
\be\label{ndm9}
\int_0^tE(\tau)d\tau \le C_1 +C_2t^{\frac12}+C_3 t^{\frac1{m+2}}.
\ee
Estimates \eqref{ndm3a} and \eqref{ndm3} imply that  $E(t)$ is nondecreasing function and therefore
$$
tE(t)\le \int_0^tE(\tau)d\tau, \ \ \forall t\ge 0.
$$
From the last inequality and \eqref{ndm9} we finally obtain the desired decay estimate
$$
E(t)\le C t^{-\frac12}, \ \ \mbox{provided} \ \ t>0 \ \ \mbox{is large enough.}
$$
\end{proof}

\section*{Acknowledgments}
{ V.K.Kalantarov\ would like to thank the Weizmann Institute of Science for the generous hospitality during which this work was initiated. E.S.Titi\ would like to thank the ICERM, Brown University, for the warm and kind hospitality where this work was completed.   The work of E.S.Titi\ was supported in part by the ONR grant N00014-15-1-2333.}
\par

\end{document}